\documentclass[12pt]{amsart}
\usepackage{amssymb,latexsym}
\usepackage{amsmath}               
\usepackage{amsfonts}       
\usepackage{amsthm}                
\usepackage{setspace}
\usepackage{amssymb}

\newtheorem{theorem}{Theorem}[section]
\newtheorem{lemma}[theorem]{Lemma}

\def\C{\mathbb{C}}
\def\N{\mathbb{N}}
\def\Z{\mathbb{Z}}

\def\Q{\mathbb{Q}}

\begin{document}


\title[On equal values of power sums of arithmetic progressions]{On equal values of power sums of arithmetic progressions}

\author[A. Bazs\'o, D. Kreso, F. Luca and \'A. Pint\'er]{A. Bazs\'o, D. Kreso, F. Luca  and \'A. Pint\'er}

\address{Institute of Mathematics \newline
\indent Number Theory Research Group of the \newline
\indent Hungarian Academy of Sciences \newline
\indent University of Debrecen \newline
\indent H-4010 Debrecen, P.O. Box 12, Hungary}
\email{bazsoa@science.unideb.hu, apinter@science.unideb.hu}

\address{Institut f\"ur Mathematik (A)\newline
\indent Technische Universit\"at Graz \newline
\indent Steyrergasse 30, 8010 Graz, Austria}
\email{kreso@math.tugraz.at}

\address{Mathematical Center UNAM\newline
\indent  UNAM Ap. Postal 61--3 (Xangari) \newline
\indent  CP 58 089, Morelia, Michoac\'an, Mexico}
\email{fluca@matmor.unam.mx}

\begin{abstract}
In this paper we consider the Diophantine equation 
\begin{align*}b^k  +\left(a+b\right)^k &+ \cdots + \left(a\left(x-1\right) + b\right)^k=\\
&=d^l + \left(c+d\right)^l + \cdots + \left(c\left(y-1\right) + d\right)^l,
\end{align*}
where $a,b,c,d,k,l$ are given integers. We prove that, under some reasonable assumptions, this equation has only finitely many integer solutions.
\end{abstract}

\maketitle

\section{Introduction and the main result}
For integers $a$ and $b$ with $\gcd (a,b) =1$ and $k, n\in \N$, $n\geq 2$, let
\begin{equation}
S_{a,b}^k \left(n\right) = b^k + \left(a+b\right)^k + \cdots + \left(a\left(n-1\right) + b\right)^k .
\end{equation}
It is easy to see that the above power sum is related to the Bernoulli polynomial $B_k (x)$ in the following way:
\begin{align*}
S_{a,b}^k \left(n\right) = \frac{a^k}{k+1}  \left( \left(B_{k+1} \left(n+ \frac{b}{a}\right)\right.\right. & -  B_{k+1}\bigg) \\
& - \left. \left(B_{k+1} \left(\frac{b}{a}\right) - B_{k+1}\right)\right),
\end{align*}
see \cite{BPS12}.
Bernoulli polynomials $B_k(x)$ are defined by the generating series 
$$\frac{t\exp(tx)}{\exp(t)-1}=\sum_{k=0}^{\infty}B_{k}(x)\frac{t^k}{k!}.$$
For the properties of Bernoulli polynomials which will be often used in this paper, sometimes without particular reference, we refer to \cite[Chapters 1 and 2]{R73}. We can extend the definition of $S_{a,b}^k(x)$ for every real value $x$ as follows
\begin{equation}\label{Sabk}
S_{a,b}^k \left(x\right) :=  \frac{a^k}{k+1} \left(B_{k+1} \left(x+ \frac{b}{a}\right) - B_{k+1} \left(\frac{b}{a}\right)\right).
\end{equation}

As usual, we denote with $\mathbb{C}[x]$ the ring of polynomials in variable $x$ with complex coefficients. If $G_1(x), G_2(x)\in \C[x]$, then $F(x)=G_1(G_2(x)) $ is a functional composition of $G_1$ and $G_2$ and $(G_1, G_2)$ is a (functional) {\it decomposition} of $F$ (over $\C$). It is said to be nontrivial if $\deg G_1 > 1$ and $\deg G_2 > 1$.
Two decompositions $F(x) = G_1 (G_2 (x))$ and $F(x) = H_1 (H_2 (x))$ are said to be {\it equivalent} if there exists a linear polynomial $\ell (x) \in \mathbb{C}[x]$ such that $G_1 (x) = H_1 (\ell(x))$ and $H_2 (x) = \ell (G_2 (x))$. The polynomial $F(x)$ is called {\it decomposable} if it has at least one nontrivial decomposition; otherwise it is said to be {\it indecomposable}.

In a recent paper, Bazs\'o, Pint\'er and Srivastava \cite{BPS12} proved the following theorem about decompositions of the polynomial $S_{a,b}^k \left(x\right)$.

\begin{theorem} \label{thm:BPS}
The polynomial $S_{a,b}^k \left(x\right)$ is indecomposable for even $k$. If $k=2v-1$ is odd, then any nontrivial decomposition of $S_{a,b}^k \left(x\right)$ is equivalent to the decomposition
\begin{equation}\label{eq:BPS}
S_{a,b}^k \left(x\right) = \widehat{S}_v \left(\left(x+\frac{b}{a} - \frac{1}{2}\right)^2\right),
\end{equation}
where $\widehat{S}_v$ is an indecomposable polynomial of degree $v$, which is uniquely determined by \eqref{eq:BPS}.
\end{theorem}

Using Theorem \ref{thm:BPS} and the general finiteness criterion of Bilu and Tichy \cite{BT00} for Diophantine equations of the form $f(x) = g(y)$, we prove the following result. 

\begin{theorem} \label{thm:BP}
For $2\leq k<l$,  the equation
\begin{equation} \label{eq:kl}
S_{a,b}^k (x) = S_{c,d}^l (y)
\end{equation}
has only finitely many solutions in  integers $x$ and $y$.
\end{theorem}

Since the finiteness criterion from \cite{BT00} is based on the ineffective theorem of Siegel \cite{S29}, Theorem \ref{thm:BP} is ineffective. We note that for $a=c=1, b=d=0$ our theorem gives the result of Bilu, Brindza, Kirschenhofer, Pint\'er and Tichy \cite{BBKPT02}.

Combining the result of Brindza \cite{B84} with recent theorems of Rakaczki \cite{R12} and Pint\'er and Rakaczki \cite{PR07}, for $k=1$ and $k=3$ we obtain effective results.

\begin{theorem}\label{Thm1}
If $l\notin\{1,3,5\}$, then integer solutions $x, y$ of the equation 
\begin{equation}\label{Thm1eq}
S_{a,b}^1 (x) = S_{c,d}^l (y)
\end{equation}
satisfy $\max\left\{|x|,|y|\right\}<C_1$, where $C_1$ is an effectively computable constant depending only on $a,b,c,d$ and $l$.
\end{theorem}

In the exceptional cases $l=3$ and $l=5$ of Theorem \ref{Thm1}, it is possible to find integers $a,b,c,d$ such that the corresponding equations have infinitely many solutions. For example, if $a=2, b=1$, $c=1, d=0$ and $l=3$, the equation \eqref{Thm1eq} becomes
$$ x^2=1+3+\cdots + 2x-1=1^3+2^3+\cdots +(y-1)^3$$
and if $l=5$ it becomes
$$ x^2=1+3+\cdots + 2x-1=1^5+2^5+\cdots +(y-1)^5.$$
These equations have infinitely many integer solutions, see \cite{S56}. 

\begin{theorem}\label{Thm3}
If $l\notin \{1,3,5\}$, then integer solutions $x, y$ of the equation 
\begin{equation}\label{Thm3eq}
S_{a,b}^3 (x) = S_{c,d}^l (y)
\end{equation}
satisfy $\max\left\{|x|,|y|\right\}<C_2$, where $C_2$ is an effectively computable constant depending only on $a,b,c,d$ and $l$.
\end{theorem}

\section{Auxiliary results}

In this section we collect some results needed to prove Theorem \ref{thm:BP}. First, we recall the finiteness criterion of Bilu and Tichy \cite{BT00}. 

We say that the equation $f(x)=g(y)$ has infinitely many {\it rational solutions with a
bounded denominator} if there exists a positive integer $\lambda$ such that
$f(x)=g(y)$ has infinitely many rational solutions $x, y$
satisfying $\lambda x, \lambda y \in \Z$. If the equation $f(x)=g(y)$ has only finitely many rational solutions with a bounded denominator, then it clearly has only finitely many integer solutions.

We further need to define five kinds of so-called {\it standard pairs} of polynomials.

In what follows $a$ and $b$ are nonzero rational numbers, $m$ and $n$ are positive integers, $r \geq 0$ is an integer and $p(x) \in \mathbb{Q} [x]$ is a nonzero polynomial (which may be constant). 

A \textit{standard pair over $\Q$ of the first kind}  is $\left(x^m, a x^{r} p(x)^m\right)$, or switched, i.e\@ $\left(a x^{r} p(x)^m, x^m\right)$, where $0 \leq r < m$, $\gcd(r,m)=1$ and $r + \deg p > 0$.

A \textit{standard pair over $\Q$ of the second kind} is $\left(x^2, \left(a x^2 + b\right) p(x)^2\right)$, or switched.

Denote by $D_{m} (x,a)$ the $m$-th Dickson polynomial with parameter $a$, defined by the functional equation $$D_{m}\left (z+\frac{a}{z}, a\right) = z^{m} + \left(\frac{a}{z}\right)^{m}$$
or by the explicit formula
\begin{equation}\label{eq:dickson}
D_{m} (x,a) = \sum_{i=0}^{\left\lfloor m / 2 \right\rfloor}{\frac{m}{m - i} \binom{m - i}{i} (-a)^i x^{m - 2i}}.
\end{equation}

A \textit{standard pair over $\Q$ of the third kind} is $\left(D_{m} \left(x,a^{n}\right), D_{n} \left(x,a^{m}\right)\right)$, where $\gcd (m, n) = 1$.

A \textit{standard pair over $\Q$ of the fourth kind} is
$$\left(a^{-m/2} D_{m} (x,a), -b^{-n/2} D_{n} (x,b)\right),$$
where $\gcd (m,n) = 2$.

A \textit{standard pair over $\Q$ of the fifth kind} is $\left(\left(a x^2 - 1\right)^3, 3x^4 - 4x^3\right)$, or switched.

The following theorem is the main result of \cite{BT00}.

\begin{theorem}\label{thm:BT}
Let $f(x)$ and $g(x)$ be non-constant polynomials in $\Q[x]$.
Then the following assertions are equivalent.
\begin{itemize}
\item[-] The equation $f(x)=g(y)$ has infinitely many
rational solutions with a bounded denominator;
\item[-] We have 
$$f(x)=\phi\left(f_{1}\left(\lambda(x)\right)\right), \quad g(x)=\phi\left(g_{1}\left(\mu(x)\right)\right),$$ 
where $\lambda(x)$ and $\mu(x)$ are linear polynomials in $\Q[x]$,
$\phi(x)\in\mathbb{Q}[x]$, and $\left(f_{1}(x),g_{1}(x)\right)$ is a
standard pair over $\Q$ such that the equation $f_1(x)=g_1(y)$
has infinitely many rational solutions with a bounded denominator.
\end{itemize}
\end{theorem}

The following lemmas are the main ingredients of the proofs of Theorems \ref{Thm1} and \ref{Thm3}

\begin{lemma} \label{eff:1}
For every $b\in \Q$ and every integer $k\geq 3$ with $k\notin \{4,6\}$, the polynomial $B_k(x)+b$ has at least three zeros of odd multiplicities.
\end{lemma}

\begin{proof}
For $b=0$ and odd values of $k\geq 3$ this result is a consequence of a theorem by Brillhart \cite[Corollary of Theorem 6]{B69}. For non-zero rational $b$ and odd $k$ with $k\geq 3$ and for even values of $k\geq 8$, the result follows from the main theorem of \cite{PR07} and from \cite[Theorem 2.3]{R12}, respectively. 
\end{proof}

Our next auxiliary result is an easy consequence of an effective theorem concerning the $S$-integer solutions of hyperelliptic equations, which is the main result of \cite{B84}.

\begin{lemma}\label{lem:hyper}
Let $f(x)$ be a polynomial with rational coefficients and with at least three zeros of odd multiplicities. Let $u$ be a fixed positive integer.
If $x$ and $y$ are integer solutions of the equation 
$$f\left(\frac{x}{u}\right)=y^2,$$
then we have $\max\left\{|x|,|y|\right\}<C_3$, where $C_3$ is an effectively computable constant depending only on $u$ and $f$.
\end{lemma}

In the sequel we assume $c_1, e_1 \in \mathbb{Q}\setminus\{0\}$ and $c_0, e_0 \in \mathbb{Q}$.

\begin{lemma} \label{lem:1}
The polynomial $S_{a,b}^k (c_1 x + c_0)$ is not of the form $e_1 x^q + e_0$ with $q \geq 3$.
\end{lemma}

\begin{lemma} \label{lem:2}
The polynomial $S_{a,b}^k (c_1 x + c_0)$ is not of the form 
$$e_1 D_{m} (x,\delta) + e_0,$$ 
where $D_{m} (x,\delta)$ is the $m$-th Dickson polynomial with $m > 4$ and $\delta \in \mathbb{Q}\setminus \{0\}$.
\end{lemma}

Before proving the lemmas above, we introduce the following notation. Let
\begin{equation}\label{Sabkc}
S_{a,b}^k (c_1 x + c_0) = s_{k+1} x^{k+1} + s_k x^k + \cdots + s_0,
\end{equation}
and $c_0 ' :=b/a + c_0$.
\
From \eqref{Sabk} we get
\begin{align}
s_{k+1}  & = \frac{a^k c_1 ^{k+1}}{k+1},\quad  s_k  =  \frac{a^k c_1 ^k}{2} (2c_0 ' - 1) \label{rel:sk}\\
s_{k-1}   & = \frac{a^k c_1 ^{k-1}}{12}k (6c_0 '^2 - 6c_0 ' + 1),\  k\geq 2, \label{rel:sk-1}
\end{align}
and for $k\geq 4$,
\begin{equation}
s_{k-3} =\frac{a^k c_1 ^{k-3}}{720}k(k-1)(k-2)(30c_0 '^4 -60c_0 '^3 + 30c_0 '^2 - 1). \label{rel:sk-3}
\end{equation}

\begin{proof}[Proof of Lemma \ref{lem:1}]
Suppose that $S_{a,b}^k (c_1 x + c_0) = e_1 x^q + e_0$, where $q=k+1 \geq 3$. Then $s_{k-1} = 0$ and from \eqref{rel:sk-1} we get $6c_0 '^2 - 6c_0 ' + 1= 0$, contradiction with $c_0 ' \in \mathbb{Q}$.
\end{proof}

\begin{proof}[Proof of Lemma \ref{lem:2}]
Suppose that $S_{a,b}^k (c_1 x + c_0) = e_1 D_{m} (x,\delta) + e_0$ with $k+1=m > 4$. Then
\begin{eqnarray}
s_{k+1}  & = &  e_1,  \quad s_k = 0, \label{rel:l2-2}\label{rel:l2-1}\\
s_{k-1} &  = & -e_1m \delta, \label{rel:l2-3}\\
s_{k-3} &= & \frac{e_1 (m -3) m \delta^2}{2}. \label{rel:l2-4} 
\end{eqnarray}
From \eqref{rel:sk} and \eqref{rel:l2-1} it follows that 
\begin{equation} \label{rel:l2-5}
e_1 = \frac{a^{m - 1} c_1 ^{m}}{m} \  \text{ and } \  c_0 ' = \frac{1}{2}.
\end{equation}
In view of \eqref{rel:sk-1}, by substituting \eqref{rel:l2-5} into \eqref{rel:l2-3}, we obtain
\begin{equation} \label{rel:l2-7}
c_1 ^2 = \frac{m -1}{24 \delta}.
\end{equation}
Similarly, by comparing the forms \eqref{rel:sk-3} and \eqref{rel:l2-4} of $s_{k-3}$ and by using \eqref{rel:l2-5}, we obtain
\begin{equation} \label{rel:l2-9}
c_1 ^4 = \frac{7(m -1)(m -2)}{2880 \: \delta^2}.
\end{equation}
After substituting \eqref{rel:l2-7} into \eqref{rel:l2-9}, we obtain $7(m -2) = 5(m -1)$, wherefrom $m = 9/2$, a contradiction. 
\end{proof}

One can see that the condition $m>4$ in Lemma \ref{lem:2} is necessary. Indeed,
\begin{align*}
S_{2,1}^2(x)&=\frac{4}{3}x^3-\frac{1}{3}x=\frac{4}{3}D_3\left(x,\frac{1}{12}\right), \\
S_{2,1}^3(x)&=2x^4-x^2=2D_4\left(x,\frac{1}{8}\right)-\frac{1}{16}.
\end{align*}

\section{Proofs of the Theorems}

\begin{proof}[Proof of Theorem \ref{Thm1}] One can rewrite the equation \eqref{Thm1eq} as
$$\frac{c^l}{l+1}\left(B_{l+1}\left(y+\frac{d}{c}\right)-B_{l+1}\left(\frac{d}{c}\right)\right)=\frac{1}{2}ax^2+\left(b-\frac{a}{2}\right)x,$$
that is
\begin{eqnarray*}
\frac{8ac^l}{l+1}\left(B_{l+1}\left(y+\frac{d}{c}\right)-B_{l+1}\left(\frac{d}{c}\right)\right)= (2ax+2b-a)^2-(2b-a)^2.
\end{eqnarray*}
Then the result follows from Lemma \ref{eff:1} and Lemma \ref{lem:hyper}.
\end{proof}

\begin{proof}[Proof of Theorem \ref{Thm3}] Using \eqref{eq:BPS} we easily see that
\begin{eqnarray*}
S_{a,b}^3 (x) & = & \frac{a^3}{4}\left(x+\frac{b}{a}-\frac{1}{2} \right)^4-\frac{a^3}{8}\left(x+\frac{b}{a}-\frac{1}{2} \right)^2\\
& + &\frac{a^4-16a^2b^2+32ab^3-16b^4}{64a}.
\end{eqnarray*}
Using the above representation, we rewrite the equation \eqref{Thm3eq} as 
$$64aS_{c,d}^l (y)+3a^4+16a^2b^2-32ab^3-16b^4=(X-2a^2)^2,$$
where $X=(2ax+2b-a)^2$. Then Lemma \ref{eff:1} and Lemma \ref{lem:hyper} complete the proof. 
\end{proof}

\begin{proof}[Proof of Theorem \ref{thm:BP}]  If the equation \eqref{eq:kl} has infinitely many integer solutions, then by Theorem \ref{thm:BT} it follows that 
$$S_{a,b}^k(a_1x+a_0)=\phi(f(x)), \quad S_{c,d}^l(b_1x+b_0)=\phi(g(x)),$$
where $(f(x), g(x))$ is a standard pair over $\Q$, $a_0, a_1, b_0, b_1$ are rationals with $a_1b_1\neq0$ and $\phi(x)$ is a polynomial with rational coefficients.

Assume that $h:=\deg \phi > 1$. Then Theorem \ref{thm:BPS} implies 
$$0<\deg f, \deg g\leq 2,$$ 
and since $k<l$ by assumption, we have $\deg f=1, \deg g=2$. Hence $k+1=h$ and $l+1=2h$, wherefrom $l=2k+1$. Since $k\geq 2$ and $l=2k+1$, it follows that $l\geq 5$. Since $\deg f=1$, there exist $f_1, f_0\in \Q$, $f_1\neq 0$, such that $S_{a, b}^k(f_1x+f_0)=\phi (x)$, so
$$S_{a, b}^k(f_1g(x)+f_0)=\phi (g(x))=S_{c, d}^l(b_1x+b_0)=S_{c, d}^{2k+1}(b_1x+b_0).$$
Since $g(x)$ is quadratic, by making the substitution $x\mapsto (x - b_0)/b_1$, we obtain that there exist $c_2, c_1, c_0\in \Q$, $c_2\neq 0$, such that
\begin{equation}\label{dec}
S_{a,b}^k(c_2x^2+c_1x+c_0)= S_{c, d}^{2k+1}(x).
\end{equation}
Since $\deg S_{a, b}^k=k+1\geq 3$ and $c_2\neq 0$, in \eqref{dec} we have a nontrivial decomposition of  $S_{c, d}^{2k+1}(x)$. From Theorem \ref{thm:BPS} it follows that there exists a linear polynomial $\ell(x)=Ax+B\in \C[x]$ such that 
$$c_2x^2+c_1x+c_0=A\left(x+\frac{d}{c}-\frac{1}{2}\right)^2+B.$$
Then clearly $A, B\in \Q$. From \eqref{dec} we obtain
$$S_{a, b}^k\left(A\left(x+\frac{d}{c}-\frac{1}{2}\right)^2+B\right)=S_{c, d}^{2k+1}(x).$$
wherefrom by linear substitution $x\mapsto x-{d}/{c}+{1}/{2}$ we obtain
\begin{equation}\label{theequation}
S_{a, b}^k(Ax^2+B)=S_{c, d}^{2k+1}\left(x-\frac{d}{c}+\frac{1}{2}\right).
\end{equation}
Thus, we have an equality of polynomials of degrees $2k+2\geq 6$. We calculate and compare coefficients of the first few highest monomials of the polynomials in \eqref{theequation}. The coefficients of the polynomial on the right-hand side are easily deduced by setting $c_1=1, c_0=-{d}/{c}+{1}/{2}$ into \eqref{rel:sk}, \eqref{rel:sk-1} and \eqref{rel:sk-3}. Therefrom it follows that if we denote 
$$S_{c,d}^{2k+1}\left(x-\frac{d}{c}+\frac{1}{2}\right)=r_{2k+2}x^{2k+2}+\cdots+ r_1x+r_0,$$
then we get
\begin{eqnarray*}
r_{2k+2} & = & \frac{c^{2k+1}}{2k+2},\\
r_{2k+1} & = & 0,\\
r_{2k} &  = & \frac{-c^{2k+1}(2k+1)}{24},\\
r_{2k-2} & = & \frac{7c^{2k+1}(2k+1)k(2k-1)}{2880}.
\end{eqnarray*}
On the other hand, the coefficients $s_{k+1}, s_k, s_{k-1}, s_{k-3}$ of the polynomial $S_{a, b}^k(x)$ can be found by setting $c_1=1, c_0=0$ into equations \eqref{rel:sk}, \eqref{rel:sk-1} and \eqref{rel:sk-3}. Since
$$S_{a, b}^k (Ax^2+B) =\sum_{m=0}^{k+1}s_m\sum_{i=0}^m\binom{m}{i}\left(Ax^2\right)^iB^{m-i},$$
it follows that if we denote
$$S_{a, b}^k (Ax^2+B)=t_{2k+2}x^{2k+2}+\cdots+t_1x+t_0,$$
then
\begin{eqnarray*}
t_{2k+2} & = & \frac{a^kA^{k+1}}{k+1},\\
t_{2k+1} &  = & 0,\\
t_{2k} & = & a^kA^kB+\frac{a^kA^k}{2}\left(2\left(\frac{b}{a}\right)-1\right),\\
t_{2k-2} & = & \frac{a^kk}{2}A^{k-1}B^2+ \frac{a^kk}{2}A^{k-1}B\left(2\left(\frac{b}{a}\right)-1\right)\\
& + &  \frac{a^kk}{12}A^{k-1}\left(6\left(\frac{b}{a}\right)^2 -6\left(\frac{b}{a}\right)+1\right).
\end{eqnarray*}
Next we compare coefficients. It must be $r_i=t_i$ for all $i=0, 1, \ldots, 2k+2$. Comparing the leading coefficients yields 
\begin{equation}
\label{eq:prva}
\frac{a^kA^{k+1}}{k+1}=\frac{c^{2k+1}}{2k+2},\qquad {\text{\rm so}}\qquad 2a^kA^{k+1}=c^{2k+1}.
\end{equation}
By comparing the coefficients of index $2k$ and using \eqref{eq:prva} we obtain
\begin{equation}
\label{druga}
\frac{b}{a}-\frac{1}{2}=-\frac{1}{12}A(2k+1)-B.
\end{equation}
By comparing the coefficients of index $2k-2$ and after simplifying we obtain
$$\frac{B^2}{2}+\frac{B}{2}\left(2\left(\frac{b}{a}\right)-1\right)+\frac{1}{12}\left(6\left(\frac{b}{a}\right)^2-6 \left(\frac{b}{a}\right)+1\right)=\frac{7(4k^2-1)A^2}{1440}.$$
From \eqref{druga} it follows that the last relation above can be transformed into 
\begin{eqnarray*}
\frac{B^2}{2}+B\left(-\frac{1}{12}A(2k+1)-B\right) & + & \frac{1}{2}\left(-\frac{1}{12}A(2k+1)-B\right)^2- \frac{1}{24}\\
& = & \frac{7A^2(4k^2-1)}{1440}.
\end{eqnarray*}
After simplification we obtain
$$A^2(k-3)(-2k-1)=15.$$
For $k\geq 3$, the expression on the left-hand side above is negative or zero, contradiction. If $k=2$, then $A^2=3$, which contradicts $A\in\Q$. 
Therefore, there are no rational coefficients $a, b, c, d, A$ and $B$ such that \eqref{theequation} is satisfied, wherefrom it follows that $\deg \phi = 1$.

If $\deg \phi=1$, then we have
$$S_{a,b} ^k (a_1 x + a_0) = e_1 f(x) + e_0, \qquad S_{c,d} ^l (b_1 x + b_0) = e_1 g(x) + e_0,$$
where $e_1, e_0 \in \mathbb{Q}$, $e_1\neq 0$. Clearly $\deg f = k+1$ and $\deg g = l+1$.

In view of the assumptions on $k$ and $l$, it follows that $(f(x),g(x))$ cannot be a standard pair over $\Q$  of the second kind, and with the exception of the case $(k,l)=(3,5)$, of the fifth kind either. If $(k,l)=(3,5)$, by using formula \eqref{rel:sk-1} for $k=3$, it is easy to see that $S_{a,b}^3(c_1x+c_0)=e_1(3x^4-4x^3)+e_0$ is not possible.

If $(f(x), g(x))$ is of the first kind, then one of the polynomials $S_{a,b} ^k (a_1 x + a_0)$ and $S_{c,d} ^l (b_1 x + b_0)$ is of the form $e_1 x^q + e_0$ with $q \geq 3$. This is impossible by Lemma \ref{lem:1}.

If $(f(x),g(x))$ is a standard pair of the third or fourth kind, then we have that either $S_{c,d} ^l (b_1 x + b_0) = e_1 D_{m} (x,\delta) + e_0$ with $m = l+1 \geq 5$ and $\delta \in \mathbb{Q}\setminus \{0\}$, which contradicts Lemma \ref{lem:2}, or $k=2, l=3$. In the latter case, Theorem \ref{Thm3} gives an effective finiteness statement. 
\end{proof}

{\bf Acknowledgements.} The authors are grateful to the referee for her/his careful reading and helpful remarks.

The research was supported in part by the Hungarian Academy of Sciences, by the OTKA grant K75566, and by the
 T\'AMOP 4.2.1./B-09/1/KONV-2010-0007 project implemented through the New Hungary Development 
Plan co-financed by the European Social Fund and the European Regional Development Fund. Dijana Kreso
was supported by the Austrian Science Fund (FWF): W1230-N13 and NAWI Graz.

\bibliographystyle{amsplain}
\bibliography{Dijana3}

\providecommand{\bysame}{\leavevmode\hbox to3em{\hrulefill}\thinspace}
\providecommand{\MR}{\relax\ifhmode\unskip\space\fi MR }
\providecommand{\MRhref}[2]{%
  \href{http://www.ams.org/mathscinet-getitem?mr=#1}{#2}
}
\providecommand{\href}[2]{#2}
\begin{thebibliography}{10}

\bibitem{BPS12}
A.~Bazs\'o, \'A. Pint\'er, and H.~M. Srivastava, \emph{A refinement of
  {F}aulhaber's theorem concerning sums of powers of natural numbers}, Applied
  Math.\@ Letters \textbf{25} (2012), 486--489.

\bibitem{BBKPT02}
Y.~Bilu, B.~Brindza, P.~Kirschenhofer, \'A. Pint\'er, and R.F. Tichy,
  \emph{Diophantine equations and {B}ernoulli polynomials. {W}ith an appendix
  by {A}. {S}chinzel}, Compositio Math.\@ \textbf{131} (2002), 173--188.

\bibitem{BT00}
Y.~Bilu and R.F. Tichy, \emph{The {D}iophantine equation $f(x) = g(y)$}, Acta
  Arith. \textbf{95} (2000), 261--288.

\bibitem{B69}
J.~Brillhart, \emph{On the {E}uler and {B}ernoulli polynomials}, J. Reine
  Angew. Math. \textbf{234} (1969), 45--64.

\bibitem{B84}
B.~Brindza, \emph{On ${S}$-integral solutions of the equation $y^m=f(x)$}, Acta
  Math. Hungar. \textbf{44} (1984), 133--139.

\bibitem{PR07}
\'A. Pint\'er and Cs. Rakaczki, \emph{On the zeros of shifted {B}ernoulli
  polynomials}, Appl.\@ Math.\@ Comput. \textbf{187} (2007), 379--383.

\bibitem{R73}
H.~Rademacher, \emph{Topics in {A}nalytic {N}umber {T}heory}, Springer-Verlag,
  1973.

\bibitem{R12}
Cs. Rakaczki, \emph{On some generalizations of the {D}iophantine equation
  $s(1^{k}+2^{k}+\cdots +x^{k})+r=dy^{n}$}, Acta Arith. \textbf{151} (2012),
  201--216.

\bibitem{S56}
J.J. Sch{\"a}ffer, \emph{The equation $1^p+2^p+3^p+\cdots+n^p=m^q$}, Acta Math.
  \textbf{95} (1956), 155 -- 189.

\bibitem{S29}
C.L. Siegel, \emph{\"{U}ber einige {A}nwendungen {D}iophantischer
  {A}pproximationes}, Abh. Preuss. Akad. Wiss. Phys.--Math. Kl. \textbf{1}
  (1929), 209 -- 266.

\end{thebibliography}

\end{document}